\documentclass[a4paper,11pt]{amsart}
\usepackage{amsmath,amsthm,amssymb,cite, graphicx}
\usepackage{wasysym}
\usepackage{color}

\renewcommand{\a}{\alpha}
\DeclareMathOperator{\RE}{Re} \DeclareMathOperator{\IM}{Im}

\numberwithin{equation}{section}
\newtheorem{theorem}{Theorem}[section]
\newtheorem{lemma}[theorem]{Lemma}

\theoremstyle{remark}

\makeatletter
\@namedef{subjclassname@2010}{%
  \textup{2010} Mathematics Subject Classification}
\makeatother

\begin{document}

\title[Starlike functions]{Radius of Starlikeness for Classes of Analytic Functions}

\author[K. Khatter]{Kanika Khatter}
\address{Department of Mathematics, SGTB Khalsa College, University of Delhi,
Delhi--110 007, India}
\email{kanika.khatter@yahoo.com}

\author[S. K. Lee]{See Keong Lee}
\address{School of Mathematical Sciences,
Universiti Sains Malaysia,
11800 USM, Penang, Malaysia}
\email{sklee@usm.my}

\author[V. Ravichandran]{V. Ravichandran}
\address{Department of Mathematics, NIT Tiruchirappalli,
Tamil Nadu--620015, India}
\email{vravi68@gmail.com}

\begin{abstract} We consider normalized analytic function $f$ on the open unit disk for which either $\RE f(z)/g(z)>0$,   $|f(z) /g(z)  - 1|<1$ or  $\RE (1-z^2) f(z)  /z>0$ for some analytic function $g$ with $\RE (1-z^2) g(z)  /z>0$.    We have obtained the radii for  these functions to   belong to various subclasses of starlike functions. The subclasses considered include the classes of starlike functions of order $\alpha$, lemniscate starlike functions and  parabolic starlike functions.
\end{abstract}

\keywords{starlike functions,  exponential function, lemniscate of Bernoulli, radius problems, coefficient estimate}

\subjclass[2010]{30C45, 30C80}

%\footnote{The authors are thankful to the referees for their useful comments.}

 \maketitle

\section{Introduction}

For any two classes $\mathcal{G}$ and $\mathcal{H}$ of analytic functions defined on the unit disk $\mathbb{D}$, the $\mathcal{H}$-radius for the class $\mathcal{G}$, denoted by $\mathcal{R}_{\mathcal{H}}(\mathcal{G})$,  is the maximal radius $\rho \leq 1$ such that the $f\in \mathcal{G}$ implies that the function $f_r$, defined by $f_r(z)=f(rz)/r$, belongs to class $\mathcal{H}$ for all $0< r \leq \rho$.   Among the radius problems for various subclasses of analytic functions, one direction of study focuses on obtaining the radius   for classes consisting of functions characterised by ratio of the function $f$ and another function $g$, where $g$ is a function belonging to some special subclass of $\mathcal{A}$ of all analytic functions on $\mathbb{D}$ normalized by $f(0)=0=f'(0)-1$.  MacGregor \cite{Mac,Mac1} obtained the radius of starlikeness for the class of functions $f \in \mathcal{A}$ satisfying either  $\RE(f(z)/g(z))>0$ or $|f(z)/g(z)-1|<1$ for some $g \in \mathcal{K}$. Ali \emph{et al.}\cite{Ali} estimated several radii for classes of functions satisfying either (i)~$\RE(f(z)/g(z))>0$, where $\RE(g(z)/z)>0$ or $\RE(g(z)/z)>1/2$; (ii) $|f(z)/g(z)-1|<1$, where $\RE(g(z)/z)>0$ or $g$ is convex; (iii) $|f'(z)/g'(z)-1|<1$, where $\RE g'(z)>0$.
The work is further investigated in \cite{asha}. These classes are related to the  Caratheodory class  $\mathcal{P}$ consisting  of all analytic functions $p$ with $p(0)=1$ and $\RE p(z)>0$ for all $z \in \mathbb{D}$.

Motivated by the aforesaid studies, we consider  the following three classes    $\mathcal{K}_1$, $\mathcal{K}_2$, and $\mathcal{K}_3$:
\begin{align*}
  \mathcal{K}_1&:= \left\{ f \in \mathcal{A} : \frac{f(z)}{g(z)}\in \mathcal{P},  ~~\text{for some}~~ g \in \mathcal{A} ,\ \RE   \frac{1-z^2}{z}g(z) >0 \right\},\\
  \mathcal{K}_2&:= \left\{ f \in \mathcal{A} : \left| \frac{f(z)}{g(z)} - 1 \right|<1, ~~ \text{for some}~~ g\in \mathcal{A} , \ \RE \frac{1-z^2}{z}g(z) >0 \right\},
  \intertext{and}
  \mathcal{K}_3&:= \left\{ f \in \mathcal{A} : \RE  \frac{1-z^2}{z}f(z) >0 \right\},
\end{align*}
and estimate the radius for the functions in the classes to belong to various subclasses of starlike functions which we discuss below.

Let $f, F$ be analytic   on $\mathbb{D}:= \{z \in \mathbb{C}: |z| < 1\}$; the function $f$ is subordinate to $F$, written $f\prec F$, provided $f=F\circ w$ for some analytic self-mapping  $w$ of the unit disk $ \mathbb{D}$ that fixes the origin. Subordination is very useful in the study of   subclasses of univalent functions. For instance, the concept of Hadamard product and subordination was used in \cite{uni} to  introduce   the class of all functions $f$ satisfying $z{(k_\a * f)'}/{(k_\a * f)}\prec h$ where $k_\a(z)= z/(1-z)^{\a}$, $\a \in \mathbb{R}$, $f \in \mathcal{A}$ and $h$ is a convex function.
Later in   1989, Shanmugam \cite{shan} studied  the class $\mathcal{S}_g^*(h)$ of all functions $f\in{\mathcal{A}}$ satisfying $z(f*g)'/(f*g) \prec h$ where $h$ is a convex function and $g$ is a fixed function in $\mathcal{A}$. By replacing $g$ with the functions $z/(1-z)$ and $z/(1-z)^2$, we get the subclasses $\mathcal{S}^*(h)$ and $\mathcal{K}(h)$ of Ma-Minda starlike and convex functions, respectively.
In 1992, Ma and Minda \cite{mam} studied the distortion, growth, covering and coefficient estimates for these functions with the weaker assumption of starlikeness on $h$.  These classes  unifies several  subclasses of starlike and convex functions. When $h$ is the mapping of $\mathbb{D}$ onto the right half-plane,  $\mathcal{S}^*(h)$ and $\mathcal{K}(h)$ reduce to the class $\mathcal{S}^*$ of starlike and $\mathcal{K}$ of convex functions, respectively. For $h(z)=(1+Az)/(1+Bz)$, with  $-1 \leq B <A \leq 1$,  the classes become $\mathcal{S}^*[A,B]$  of Janowski starlike functions and $\mathcal{K}[A,B]$ of Janowski convex functions.  For $A= 1- 2 \alpha$ and $B = -1$ where $0 \leq \a < 1$,  these subclasses become $\mathcal{S}^*(\a)$ of the starlike functions of order $\a$  and $\mathcal{K}(\alpha)$ of   convex functions of order $\a$, respectively  introduced   by Robertson \cite{rob}.  For $h (z)= \sqrt{1+z}$,  the class $\mathcal{S}^*(h)$ becomes the class $\mathcal{S}^*_{L} $ of the lemniscate starlike functions introduced and studied by Sok\'{o}l and Stankiewicz \cite{sok,sok2}; analytically,  $f \in \mathcal{S}^*_{L}$  if  $|(z f'(z)/f(z))^2-1|<1$.

Mendiratta \emph{et al.}\cite{sumit,sumit2} studied the classes $\mathcal{S}^*_e = \mathcal{S}^*(e^z)$ and $\mathcal{S}^*_{RL} = \mathcal{S}^*(h_{RL})$, where
\[ h_{RL}:= \sqrt{2} - (\sqrt{2} -1) \sqrt{\frac{1-z}{1+2(\sqrt{2}-1)z}}.\] Indeed, a  function $f$ belongs to $\mathcal{S}^*_e$ or to $ \mathcal{S}^*_{RL}$ if $z f'(z)/f(z)$ respectively belongs to  $\{ w \in \mathbb{C}: |\log w|<1\}$  or $\{ (w - \sqrt{2})^2 -1 < 1 \}$. Sharma \emph{et al.}\cite{sharma} defined and studied the class of functions defined by $\mathcal{S}^*_c = \mathcal{S}^* (h_c(z))$, where $h_c(z) = 1 + (4/3) z + (2/3) z^2$; a function $f\in \mathcal{S}^*_c$ if  $z f'(z)/f(z)\in \{ x + i y : (9 x^2 + 9 y^2 -18 x + 5)^2 -16 (9x^2 + 9 y^2 -6 x +1) = 0 \}$. Cho \emph{et al.}\cite{cho} defined and studied the class $\mathcal{S}^*_{\sin} = \mathcal{S}^*(1 + \sin z)$. Raina and Sokol \cite{raina} defined the class $\mathcal{S}^*_{\leftmoon} = \mathcal{S}^*(h_{\leftmoon})$, where $h_{\leftmoon} = z + \sqrt{1+ z^2}$ and $\mathcal{S}^*_{\leftmoon}$ consists of functions for which $z f'(z)/f(z)$ lies in the the leftmoon region defined by $\Omega_{\leftmoon}= h_{\leftmoon}(\mathbb{D}) :=\{ w\in \mathbb{C}: |w^2 - 1| < 2 |w| \}$. Another particular case is the class $\mathcal{S}^*_R = \mathcal{S}^*(h_R)$ studied in \cite{kumar} where $h_R = 1+ (zk + z^2)/(k^2 - kz)$, and $k = \sqrt{2} + 1$.   The subclass $\mathcal{S}_{P} $ of  parabolic starlike functions (see the survey \cite{ronning} or \cite{ali,mam1, ganga}) consists of all normalized analytic  functions $f$   with $zf'(z)/f(z)$ lying  in the parabolic region $(\IM(w))^2< 2 \RE(w)-1$.

\section{Main Results}
The first theorem  gives the various radii of starlikeness for the class $\mathcal{K}_1$ which consists of functions $f \in \mathcal{A}$ satisfying $\RE(f(z)/g(z)) > 0$ for some $g \in \mathcal{A}$ and $\RE  ((1-z^2)g(z)/z) >0$.
Note that the functions $f_1,\ g_1: \mathbb{D}\rightarrow \mathbb{C}$ defined by
\begin{equation} \label{f1}
  f_1(z) = \frac{z(1+ i z)^2}{(1-z^2)(1-i z)^2} \quad \quad ~~~ \quad \text{and}~~~ \quad \quad g_1(z)= \frac{z (1+ i z)}{(1-z^2) (1- i z)}
\end{equation}
satisfy
\begin{equation*}
  \RE  \frac{f_1(z)}{g_1(z)}  = \RE  \frac{1-z^2}{z}g_1(z)  = \RE  \frac{1+ i z}{1- i z} >0.
\end{equation*}
This means the function $f_1 \in \mathcal{K}_1$ and so $\mathcal{K}_1 \neq \phi$. Further we will see that this function $f_1$ serves as an extremal function for many radii problems studied here.

\begin{theorem}\label{th1}
For the class $\mathcal{K}_1$, the following results hold:
\begin{enumerate}
  \item The $\mathcal{S}^*(\a)$-radius is the smallest positive real root of the equation $r^4(1+\a) - 4 r^3-2 r^2 -4 r +(1-\a)=0$, \quad $0 \leq \a <1$.
  \item The $\mathcal{S}^*_L$-radius is ${R}_{\mathcal{S}^*_L}= (\sqrt{5}-2)/(\sqrt{2}+1) \approx 0.0977826$. %This radius obtained is sharp.
  \item The  $\mathcal{S}_{P}$-radius is the smallest positive real root of the equation $ 3 r^4 - 8r^3 -4 r^2 -8 r +1 = 0 $ i.e. ${R}_{\mathcal{S}_{P}} \approx 0.116675$.
  \item The  $\mathcal{S}^*_{e}$-radius is the smallest positive real root of the equation $(2r^2 +4 r +4 r^3 -1-r^4)e = r^4 - 1 $ i.e. ${R}_{\mathcal{S}^*_{e}} \approx 0.144684$.
  \item The $\mathcal{S}^*_{c}$-radius is the smallest positive real root of the equation $ 4 r^4 - 12 r^3 - 6 r^2 - 12 r + 2 = 0 $ i.e. ${R}_{\mathcal{S}^*_{c}} \approx 0.15182$.

  \item The $\mathcal{S}^*_{\leftmoon}$-radius is the smallest positive real root of the equation $ 4 r^3 + 2 r^2 + 4 r + \sqrt{2}(1- r^4) = 2 $ i.e. ${R}_{\mathcal{S}^*_{\leftmoon}} \approx 0.134993$.
  \item The $\mathcal{S}^*_{\sin}$-radius is ${R}_{\mathcal{S}^*_{\sin}}= (-2+\sqrt{4+ \sin{1}(2+ \sin{1})})/(2+ \sin{1}) \approx 0.185835$.

  \item The $\mathcal{S}^*_{RL}$-radius is ${R}_{\mathcal{S}^*_{RL}} \approx 0.0687813$.
\item The $\mathcal{S}^*_{R}$-radius is the smallest positive real root of the equation $ 4 r^3  + 2 r^2 + 4 r - r^4 -1 = 2 (1- \sqrt{2})(1- r^4) $ i.e. ${R}_{\mathcal{S}^*_{R}} \approx 0.0419413$.
\end{enumerate}
All the radii obtained are sharp.
\end{theorem}

We would use the following lemmas in order to prove our results:
\begin{lemma}[{\cite[Lemma 2.2, p.\ 4]{jain}}]\label{T1}
 For $0<a<\sqrt{2}$, let $r_a$ be given by
 \begin{align*}
   r_a= \left\{
          \begin{array}{ll}
            (\sqrt{1-a^2}- (1-a^2))^{1/2}, & \hbox{$0< a \leq 2 \sqrt{2}/3$;} \\
            \sqrt{2}-a, & \hbox{$ 2 \sqrt{2}/3 \leq a < \sqrt{2}$.}
          \end{array}
        \right.
 \end{align*}
Then $\{w: |w-a|< r_a \} \subseteq \{w: |w^2-1|<1\}$.
\end{lemma}

\begin{lemma}[{\cite[Lemma 1, p.\ 321]{shan}}]\label{T2}
 For $a> 1/2$, let $r_a$ be given by
 \begin{align*}
   r_a= \left\{
          \begin{array}{ll}
            a-1/2, & \hbox{$1/2< a \leq 3/2$;} \\
            \sqrt{2a-2}, & \hbox{$ a \geq 3/2$.}
          \end{array}
        \right.
 \end{align*}
Then $\{w: |w-a|< r_a \} \subseteq \{w: \RE w > |w-1|\} =  \Omega_{1/2}$. Here, $\Omega_p$ is a parabolic region which is symmetric with respect to the real axis and vertex at $(p,0)$.
\end{lemma}

\begin{lemma}[{\cite[Lemma 2.2, p.\ 368]{sumit}}]\label{T3}
 For $ e^{-1} < a < e$, let $r_a$ be given by
 \begin{align*}
   r_a= \left\{
          \begin{array}{ll}
            a-e^{-1}, & \hbox{$e^{-1}< a \leq  (e+e^{-1})/2$;} \\
            e - a , & \hbox{$ (e+e^{-1})/2 \leq a < e$.}
          \end{array}
        \right.
 \end{align*}
Then $\{w: |w-a|< r_a \} \subseteq \{w: |\log w|<1\} =  \Omega_e$, which is the image of the unit disk $\mathbb{D}$ under the exponential function.
\end{lemma}

\begin{lemma}[{\cite[Lemma 2.5, p.\ 926]{sharma}}]\label{T4}
 For $ 1/3 < a < 3$, let $r_a$ be given by
 \begin{align*}
   r_a= \left\{
          \begin{array}{ll}
            (3a-1)/3, & \hbox{$1/3< a \leq 5/3$;} \\
            3- a , & \hbox{$  5/3 \leq a \leq 3$.}
          \end{array}
        \right.
 \end{align*}
Then $\{w: |w-a|< r_a \} \subseteq  \Omega_c$.
Here $\Omega_c$ is the region bounded by the cardioid $\{x+ i y: (9x^2+9y^2-18x+5)^2 - 16 (9 x^2+ 9 y^2 -6x + 1) = 0 \}$.
\end{lemma}

\begin{lemma}[{\cite[Lemma 3.3, p.\ 7]{cho}}]\label{T5}
 For $ 1 - \sin 1 < a < 1+ \sin 1 $, let $r_a= \sin 1 - |a-1|$. Then $\{w: |w-a|< r_a \} \subseteq  \Omega_{sin}$. Here $\Omega_{sin}$ is the image of the unit disk $\mathbb{D}$ under the function $1+ \sin z$.
\end{lemma}
\begin{lemma}[{\cite[Lemma 2.1, p.\ 3]{gandhi}}]\label{T7}
 For $ \sqrt{2}-1 < a < \sqrt{2}+1$, let $r_a=  1 - |\sqrt{2}-a|$. Then $\{w: |w-a|< r_a \} \subseteq  \Omega_{\leftmoon}= \{w: |w^2 - 1|< 2 |w|\}$.
\end{lemma}
\begin{lemma}[{\cite[Lemma 2.2, p.\ 202]{kumar}}]\label{T8}
 For $ 2 (\sqrt{2}-1) < a < 2$, let $r_a$ be given by
 \begin{align*}
   r_a= \left\{
          \begin{array}{ll}
            a-2(\sqrt{2}-1), & \hbox{$2 (\sqrt{2}-1) < a \leq \sqrt{2}$;} \\
            2- a , & \hbox{$\sqrt{2} \leq a < 2$.}
          \end{array}
        \right.
 \end{align*}
Then $\{w: |w-a|< r_a \} \subseteq  \Omega_R$, where $\Omega_R$ is the image of the unit disk $\mathbb{D}$ under the function $1+ ((z k + z^2)/(k^2 -k z))$, ~~ $k= \sqrt{2}+1$.
\end{lemma}

\begin{lemma}[{\cite[Lemma 3.2, p.\ 10]{sumit2}}]\label{T9}
 For $ 0 < a < \sqrt{2}$, let $r_a$ be given by
 \begin{align*}
   r_a= \left\{
          \begin{array}{ll}
            a, & \hbox{$0 < a \leq \sqrt{2}/3$;} \\
           \big((1- (\sqrt{2}-a)^2)^{1/2}-(1-(\sqrt{2}-a)^2)\big)^{1/2}, & \hbox{$\sqrt{2}/3 \leq a < \sqrt{2}$.}
          \end{array}
        \right.
 \end{align*}
Then $\{w: |w-a|< r_a \} \subseteq  \{ w: \RE w >0, |(w- \sqrt{2})^2-1|< 1 \}= \Omega_{RL}$.
\end{lemma}

\begin{lemma}[{\cite[Lemma 2, p.\ 240]{shah}}]\label{T10}
 If $ p(z) = 1 + b_n z^n + b_{n+1} z^{n+1} + \cdots $ is analytic and satisfies $ \RE p(z) > \alpha$, $0 \leq \alpha < 1$,  for $|z|< 1$, then
\begin{equation*}
  \left| \frac{zp'(z)}{p(z)}\right| \leq \frac{2 n z^n(1-\alpha)}{(1- |z|^n)(1+ (1-2 \alpha)|z|^n)}.
\end{equation*}
\end{lemma}

With all these tools, we are ready to give the proof of our first result.

\begin{proof}[Proof of Theorem~\ref{th1}]
Let $f \in \mathcal{K}_1$ and the function $g: \mathbb{D} \rightarrow \mathbb{C}$ be chosen such that
\begin{equation}\label{1}
  \RE {\frac{f(z)}{g(z)}}>0 \quad \text{and} \quad \RE{\Big( \frac{1-z^2}{z} g(z) \Big)}>0 \quad (z \in \mathbb{D}).
\end{equation}
Let us define $p_1, p_2: \mathbb{D}\rightarrow \mathbb{C}$ as
\begin{equation}\label{2}
  p_1(z)= \frac{1-z^2}{z} g(z) \quad \text{and} \quad p_2(z)= \frac{f(z)}{g(z)}
\end{equation}
Therefore, by equation \eqref{1}, %and \eqref{2},
$p_1$ and $p_2$ are in $\mathcal{P}$. Equation \eqref{2} yields
\begin{equation*}
  f(z) = \frac{z}{(1-z^2)} p_1(z) p_2(z).
\end{equation*}
Take logarithm at both sides and differentiate with respect to $z$ would give
\begin{equation}\label{3}
  \frac{z f'(z)}{f(z)}= \frac{1+z^2}{1-z^2}+ \frac{z p_1'(z)}{p_1(z)}+ \frac{z p_2'(z)}{p_2(z)}.
\end{equation}
It can be easily proved that the bilinear transform $w= (1+z^2)/(1-z^2)$ maps the disk $|z|\leq r$ onto the disk
\begin{equation}\label{4}
  \left| \frac{1+z^2}{1-z^2}- \frac{1+r^4}{1-r^4} \right| \leq \frac{2r^2}{1-r^4}.
\end{equation}
Now, by Lemma \ref{T10}, for $p \in \mathcal{P}(\a) := \{ p \in \mathcal{P} : \RE p(z) > \a, z \in \mathbb{D}\}$, we have
\begin{equation}\label{5}
  \left| \frac{z p'(z)}{p(z)}\right|  \leq \frac{2 (1- \a) r}{(1-r)\big(1+ (1-2\a)r\big)} \quad  (|z|\leq r).
\end{equation}
By using equations \eqref{3}, \eqref{4} and \eqref{5}, we can conclude that a function $f \in \mathcal{K}_1$ maps the disk $|z|\leq r$ onto the disk
\begin{equation}\label{6}
  \left| \frac{z f'(z)}{f(z)} - \frac{1+r^4}{1-r^4} \right| \leq  \frac{2r (2 r^2 + r + 2)}{1-r^4}.
\end{equation}
In order to solve radius problems for $f \in \mathcal{K}_1$, we are interested in computing the value of $r$ for which the disk in \eqref{6} is contained in the corresponding regions.
The classes we are considering here are all subclasses of starlike functions and therefore, we first determine the radius of starlikeness for $f \in \mathcal{K}_1$. From \eqref{6}, we have
\begin{equation*}
  \RE \frac{z f'(z)}{f(z)} \geq \frac{r^4 - 4 r^3 - 2 r^2 -4r + 1}{1-r^4} \geq 0.
\end{equation*}
Solving the above inequality for $r$, we get that the function $f \in \mathcal{K}_1$ is starlike in $|z|\leq 0.216845$. Hence, all the radii that we are going to estimate here, will be less than 0.216845. For the function $f_1$ defined in \eqref{f1}, we have
\begin{align*}
 \frac{zf'_1(z)}{f_1(z)} & = \frac{1+ 4 i z +2 z^2 -4 i z^3 +z^4}{1- z^4}\\
                     & = \frac{1+ 4 i z(1- z^2) + 2 z^2 + z^4 }{1 - z^4}
\end{align*}
At $z := r i = (0.216845) i $, we have $zf'_1(z)/f_1(z)\approx 0$, thereby proving that the radius of starlikeness obtained for the class $\mathcal{K}_1$ is sharp.
\begin{enumerate}
  \item In order to compute $R_ {\mathcal{S}^*(\a)}$, we estimate the value of $r \in (0,1)$ satisfying
   \begin{equation*}
  \RE \frac{z f'(z)}{f(z)} \geq \frac{r^4 - 4 r^3 - 2 r^2 -4r + 1}{1-r^4} \geq \a.
\end{equation*}
      Therefore, the number $r = R_ {\mathcal{S}^*(\a)}$, is the smallest positive real root of the equation $r^4(1+\a) - 4 r^3-2 r^2 -4 r +(1-\a)=0$ in $(0,1)$. For the function $f_1 \in \mathcal{K}_1$ given by \eqref{f1}, we have
\begin{align} \label{f11}
  \frac{z f_1'(z)}{f_1(z)} & = \frac{1+ 4 i z +2 z^2 -4 i z^3 +z^4}{1- z^4}
\end{align}
At $z:= r i = \mathcal{R}_{\mathcal{S}^*(\alpha)}$,  \eqref{f11} reduces to
\begin{align*}
  \frac{z f_1'(z)}{f_1(z)} & = \frac{1- 4 r - 2 r^2 -4  r^3 +r^4}{1- r^4}
= \alpha,
\end{align*}
thereby proving that the radius is sharp.
\item We use lemma \ref{T1} to compute the lemniscate starlike radius for the function $f \in \mathcal{K}_1$.
    Let $a = (1+r^4)/(1-r^4)$. Then for $0 \leq r < 1$, we have $a \geq 1$. So for $a < \sqrt{2}$, we get $r < \sqrt[4]{(\sqrt{2}-1)/(\sqrt{2}+2)} \approx 0.59018$. On the other hand, consider
    \[
    \frac{2r(2r^2 + r + 2)}{1-r^4} \leq \sqrt{2} - a = \sqrt{2} - \frac{1+r^4}{1 - r^4}.
    \]
    From this, let $r^*$ be the smallest positive real roof of the equation $(1+\sqrt{2})r^4 + 4r^3 + 2r^2 + 4r + (1 - \sqrt{2}) = 0$. Then the radius of lemniscate starlikeness for $f\in \mathcal{K}_1$ is
    \[
    R_{\mathcal{S}^*_L} = \min\left\{ \left( \frac{\sqrt{2}-1}{\sqrt{2}+2}\right)^{1/4}, r^*\right\} = r^* = \frac{\sqrt{5}-2}{\sqrt{2}+1}.
    \]

The radius obtained is sharp. Consider the functions $f, g: \mathbb{D}\rightarrow \mathbb{C}$ defined by
\begin{equation} \label{f_1}
  f(z) = \frac {z(1-z)}{(1+z)^3} \quad \quad ~~~ \quad \text{and}~~~ \quad \quad g(z)= \frac{z}{(1+z)^2}.
\end{equation}
Then clearly $f \in \mathcal{K}_1$ as
\begin{equation*}
  \RE  \frac{f(z)}{g(z)}  = \RE  \frac{1-z^2}{z}g(z)  = \RE  \frac{1+  z}{1-  z} >0.
\end{equation*}

Now, for $z := - r^* = - R_{\mathcal{S}^*_L}$, we have $(z^2 -4 z +1)/(1-z^2) = \sqrt{2}$ and thus
\begin{equation*}
\left| \left(\frac{zf'(z)}{f(z)}\right)^2 -1  \right| = \left| \left(\frac{r^2 -4r+1}{1-r^2}\right)^2 -1 \right|=1,
\end{equation*}
thereby proving that the radius obtained is sharp by the function $f$ in \eqref{f_1}.

\item We use Lemma \ref{T2} to compute the parabolic starlike radius for $f \in \mathcal{K}_1$. Again, let $a = (1+r^4)/(1-r^4)$, which is larger than or equal to $1$ for $0 \leq r < 1$. Note that
    \[
    a = \frac{1 + r^4}{1 - r^4} = \frac{3}{2} \quad \Leftrightarrow \quad r= \left(\frac{1}{5} \right)^{1/4} \approx 0.66874.
    \]
    Since the radius we are looking for would be less than $0.216845$, we only consider the case $1/2 < a \leq 3/2$ in Lemma \ref{T2}. So when considering
    \[
    \frac{2r(2r^2 + r + 2)}{1-r^4} \leq \frac{1 + r^4}{1-r^4} - \frac{1}{2},
    \]
    let $r^*$ be the smallest positive real root of the equation $3 r^4-8 r^3 -4 r^2 -8 r +1= 0$. Then the radius of parabolic starlikeness for $f\in \mathcal{K}_1$ is
    \[
    R_{\mathcal{S}_P} = \min\left\{ \left(\frac{1}{5}\right)^{1/4} , r^*\right\} = r^* \approx 0.116675.
    \]

\indent We see that the sharpness follows for the function $f_1 \in \mathcal{K}_1$ defined in \eqref{f1}. At $z = ir$, we have
\begin{equation*}
  F(r) = \frac{z f_1'(z)}{f_1(z)}\bigg|_{z = ir}  = \frac{1- 4r - 2 r^2 -4  r^3 +r^4}{1- r^4}.
\end{equation*}
Then,
\begin{align*}
\left| F(r)- 1  \right| &= \left| \frac{2r (r^3-2r^2-r-2)}{1-r^4}\right|.%\\
                                           %&=  \frac{2r (2+r+2r^2-r^3)}{1 - r^4}
\end{align*}
For $z := ir^* = iR_{\mathcal{S}_{P}}$, we have
\begin{align*}
  \RE  \frac{z f'_1(z)}{f_1(z)}  &= \frac{1+ r^4 - 4 r^3 -2 r^2- 4r}{1- r^4} (\approx 0.5)\\
&= \frac{2 r (2+r+2r^2-r^3)}{1 -r^4}=  \left| \frac{zf_1'(z)}{f_1(z)}- 1  \right|.
\end{align*}
Thus the radius obtained is sharp for the function $f_1$.

\item By using  Lemma \ref{T3} and the argument similar to the above, we get that the exponential starlike radius $R_{\mathcal{S}^*_e}$ for the class $\mathcal{K}_1$ is the smallest positive real root of the equation $(4 r^3 +2 r^2 + 4 r  -1 -r^4)e = r^4 -1$.

The radius is sharp for the function $f_1$ defined in \eqref{f1}.
For $z := ir = i{R}_{\mathcal{S}^*_e}$, we have
\begin{equation*}
  \left|\log \frac{zf'_1(z)}{f_1(z)} \right| = \left| \log \frac{1 + r^4 - 4r^3 - 2 r^2 - 4r}{1 - r^4}\right|=1.
\end{equation*}

\item
By using Lemma \ref{T4}, and similar argument as before,
the $\mathcal{S}^*_c$-radius for the class $\mathcal{K}_1$ is the smallest positive real root of the equation $2 r^4 - 6 r^3 - 3 r^2 - 6 r +1 = 0$.

The radius is sharp for the function $f_1$ defined in \eqref{f1}.
% and it can be seen from Figure~\ref{fig1}.
%\begin{figure}[h]
%  \includegraphics[width=4cm, height=5cm]{1}
%  \caption{Sharpness of cardioid starlikeness} \label{fig1}
%\end{figure}
Indeed, for the function $f_1$ defined in \eqref{f1}, we have
at $z := ir = i \mathcal{R}_{\mathcal{S}^*_c}$,
\begin{equation*}
    \frac{zf'_1(z)}{f_1(z)} =  \frac{1 + r^4 - 4r^3 - 2 r^2 - 4r}{1 - r^4} =\frac{1}{3}=h_c(-1)\in \partial h_c(\mathbb{D}),
\end{equation*}
where $h_c(z) = 1 + (4/3)z + (2/3)z^2$.
This shows that the result is sharp.

\item To determine the $\mathcal{S}^*_{\leftmoon}$-radius, ${R}_{\mathcal{S}^*_{\leftmoon}}$, we will use Lemma \ref{T7}. After some computations following the idea above, it can be shown that ${R}_{\mathcal{S}^*_{\leftmoon}}$ is the smallest positive real root of the equation $4 r^3 + 2 r^2 + 4r = 2 -\sqrt{2}(1-r^4)$.

The radius is sharp for the function $f_1$ defined in \eqref{f1}, since at $z := ir = i{R}_{\mathcal{S}^*_{\leftmoon}}$, we have
\begin{align*}
  \left|  \left( \frac{zf_1'(z)}{f_1(z)}\right)^2-1 \right| &= \left| \left( \frac{1+ r^4 - 4 r^3 - 2 r^2 - 4 r }{1-r^4}\right)^2 -1 \right| (\approx 0.134993 )\\&= 2\left| \frac{1+ r^4 - 4 r^3 - 2 r^2 - 4 r }{1-r^4} \right| = 2 \left|\frac{zf_1'(z)}{f_1(z)} \right|.
\end{align*}

\item In order to find the $\mathcal{S}^*_{\sin}$-radius for function $f \in \mathcal{K}_1$, we make use of Lemma \ref{T5}. Similarly as above, with $a = {(1+r^4)}/{(1-r^4)} >1$, it can be shown by arguing similarly as above that the $\mathcal{S}^*_{\sin}$-radius is the smallest positive real root of the equation $(2 + \sin{1})r^4 + 4r^3 + 2r^2 + 4r - \sin{1} = 0$.

The radius is sharp for the function $f_1$ defined in \eqref{f1}. %and it can be seen from Figure \ref{fig2}. The details are omitted.
%\begin{figure}[h]
%  \includegraphics[width=4.5cm, height=5cm]{2}
% \caption{Sharpness of sine starlikeness} \label{fig2}
%\end{figure}

\item In order to compute the $\mathcal{S}^*_{RL}$- radius for the class $\mathcal{K}_1$, we use Lemma \ref{T9}. As $\sqrt{2}/3  \leq a = {(1+ r^4)}/{(1- r^4)} < \sqrt{2}$, a computation using Lemma \ref{T9} shows that the $\mathcal{S}^*_{RL}$- radius is the smallest positive real root of the equation
    \[
    4r^2(2r^2 + r + 2)^2 = (1-r^4)\sqrt{\left(\sqrt{2}-1\right) + \left(\sqrt{2}-2\right)r^4} - 2\left( \sqrt{2}-1 + (\sqrt{2}-2)r^4\right).
    \]

The radius obtained is sharp for the function $ f \in \mathcal{K}_1$ given by \eqref{f_1}. % Consider,

At $ z := - r  = - {R}_{\mathcal{S}^*_{RL}}$, we have $(z^2 -4 z +1)/(1-z^2) = \sqrt{2}$ and therefore,
\begin{equation*}
   \left| \left( \frac{zf'(z)}{f(z)} - \sqrt{2} \right)^2 -1 \right| = \left| \left( \frac{1- 4z + z^2}{1- z^2}  - \sqrt{2}  \right)^2 -1 \right| = 1.
\end{equation*}
Hence the result.

\item   Since $2 (\sqrt{2}-1) < a = {(1+ r^4)}/{(1- r^4)}\leq \sqrt{2}$, by using Lemma \ref{T8}, it can be shown that the $\mathcal{S}_R$- radius is obtained by solving the equation
     \begin{equation*}
     \left( 2\sqrt{2} - 1\right)r^4 - 4r^3 - 2r^2 - 4r + \left( 3 - 2\sqrt{2}\right)=0.
     \end{equation*}

The radius is sharp for the function $f_1$ defined in \eqref{f1}.% and it can be seen by the following Figure \ref{fig3}:
%\begin{figure}[h]
%  \includegraphics[width=7.5cm, height=6cm]{3}
% \caption{Sharpness of rational- function starlikeness} \label{fig3}
%\end{figure}

Indeed, for the function $f_1$ defined in \eqref{f1}, we have
at $z := ir = i {R}_{\mathcal{S}^*_R}$ that
\begin{equation*}
    \frac{zf'_1(z)}{f_1(z)} =  \frac{1 + r^4 - 4r^3 - 2 r^2 - 4r}{1 - r^4} = 2\sqrt{2} -2 =h_R(-1)\in \partial h_R(\mathbb{D}).
\end{equation*}
Here, $h_R = 1+ (zk + z^2)/(k^2 - kz)$, and $k = \sqrt{2} + 1$. %This shows that the result is sharp.
\qedhere
\end{enumerate}
\end{proof}

Our next result gives various radii of starlikeness for the $\mathcal{K}_2$, which consists of functions $f \in \mathcal{A}$ satisfying $|(f(z)/g(z))-1|<1$ for some $g \in \mathcal{A}$ and $\RE  ((1-z^2)g(z)/z) >0$.
Consider the functions $f_2$, $g_2: \mathbb{D}\rightarrow \mathbb{C}$ defined by
\begin{equation} \label{f2}
  f_2(z) = \frac{z(1+ i z)^2}{(1-z^2)(1-i z)} \quad \quad ~~~ \quad \text{and}~~~ \quad \quad g_2(z)= \frac{z (1+ i z)}{(1-z^2) (1- i z)}.
\end{equation}
Clearly,
\begin{equation*}
 \left|  \frac{f_2(z)}{g_2(z)} - 1 \right| = |i z| = |z|< 1~~~~\quad  \text{and}~~~~\quad \RE  \frac{1-z^2}{z}g_2(z)  = \RE  \frac{1+ i z}{1- i z} >0.
\end{equation*}
Therefore, the function $f_2$ is in $\mathcal{K}_2$ and this shows  $\mathcal{K}_2 \neq \phi$. Note that this function $f_2$ would serve as an extremal function for several radii-problems that we study here.
\begin{theorem}
  For $f \in \mathcal{K}_2$, the following results hold:
\begin{enumerate}
  \item The sharp $\mathcal{S}^*(\a)$ radius is the smallest positive real root of the equation $ \a r^4 -3r (r^2 + r + 1) + (1- \a)=0$, \quad $0 \leq \a <1$.
  \item The $\mathcal{S}^*_L$ radius is ${R}_{\mathcal{S^*_L}}= (\sqrt{2}-1)/(\sqrt{2}+2) \approx 0.12132$.
  \item The sharp $\mathcal{S}_{P}$ radius is the smallest positive real root of the equation $ 6 r^3 + 6 r^2 + 6 r -1  - r^4   = 0$ i.e., ${R}_{\mathcal{S}_{P}} \approx 0.1432698$.
  \item The sharp $\mathcal{S}^*_{e}$ radius is the smallest positive real root of the equation $(3 r^3 + 3 r^2 + 3 r -1)e + 1 - r^4 = 0$ i.e.,  ${R}_{\mathcal{S}^*_{e}} \approx 0.174887$.
  \item The sharp $\mathcal{S}^*_{c}$ radius is the smallest positive real root of the equation $  9 r^3 + 9 r^2 + 9 r - 2 - r^4 = 0 $ i.e., ${R}_{\mathcal{S}^*_{c}} \approx 0.182815$.

  \item The sharp $\mathcal{S}^*_{\leftmoon}$ radius is the smallest positive real root of the equation $ r^4 (1- \sqrt{2}) + 3 r^3 + 3 r^2 + 3 r = 2 - \sqrt{2}$ i.e., ${R}_{\mathcal{S}^*_{\leftmoon}} \approx 0.164039$.
 \item The sharp $\mathcal{S}^*_{\sin}$ radius is ${R}_{\mathcal{S}^*_{\sin}}= \sin{1}/(3+ \sin{1}) \approx 0.219049$.
 \item The sharp $\mathcal{S}^*_{R}$ radius is the smallest positive real root of the equation $ 2 r^4 + 3 r^3 + 3 r^2 +  3 r  - 3 + 2 \sqrt{2}(1- r^4) =0  $ i.e., ${R}_{\mathcal{S}^*_{R}} \approx 0.0541073$.
  \item The $\mathcal{S}^*_{RL}$ radius is ${R}_{\mathcal{S}^*_{RL}} \approx 0.0870259$.
\end{enumerate}
\end{theorem}

\begin{proof}
Let $f \in \mathcal{K}_2$ and the function $g: \mathbb{D} \rightarrow \mathbb{C}$ be chosen such that
\begin{equation}\label{2.1}
  \left| {\frac{f(z)}{g(z)}}-1\right|< 1 \quad \text{and} \quad \RE{\Big( \frac{1-z^2}{z} g(z) \Big)}>0 \quad (z \in \mathbb{D}).
\end{equation}
Note that $| {f(z)}/{g(z)}-1|< 1 $ holds if and only if $\RE (g(z)/f(z)) > 1/2$
Let define $p_1, p_2: \mathbb{D}\rightarrow \mathbb{C}$ as
\begin{equation}\label{2.2}
  p_1(z)= \frac{1-z^2}{z} g(z) \quad \text{and} \quad p_2(z)= \frac{g(z)}{f(z)}.
\end{equation}
Then, by equations \eqref{2.1} and \eqref{2.2}, $p_1 \in \mathcal{P}$ and  $p_2 \in \mathcal{P}(1/2)$. Equation \eqref{2.2} also yields
\begin{equation*}
  f(z) = \frac{z}{1-z^2} \frac{p_1(z)}{ p_2(z)}.
\end{equation*}
Taking logarithm on both sides and differentiating with respect to $z$ gives
\begin{equation}\label{2.3}
  \frac{z f'(z)}{f(z)}= \frac{1+z^2}{1-z^2}+ \frac{z p_1'(z)}{p_1(z)}- \frac{z p_2'(z)}{p_2(z)}.
\end{equation}

By using equations \eqref{4},\eqref{5} and \eqref{2.3}, it can proven that the function $f$ maps the disk $|z|\leq r$ onto the disk
\begin{equation}\label{2.6}
  \left| \frac{z f'(z)}{f(z)} - \frac{1+r^4}{1-r^4} \right| \leq  \frac{r (r^3 + 3 r^2 + 3 r + 3)}{1-r^4}.
\end{equation}
From \eqref{2.6}, we can get
\begin{equation*}
  \RE \frac{z f'(z)}{f(z)} \geq  \frac{1 - 3 r ( r^2 +  r + 1)}{1-r^4} \geq 0.
\end{equation*}
Upon solving for $r$, we can conclude that the function $f \in \mathcal{K}_2$ is starlike in $|z|\leq 0.253077\cdots$.
The classes we are considering here are all subclasses of starlike functions, hence, all the radii we estimate here, will be less than $0.253077\cdots$. For the function $f_2$ defined in \eqref{f2}, we have
\begin{align*}
 \frac{zf'_2(z)}{f_2(z)} & = \frac{1 + 3 i z + 3 z^2 - 3 i z^3}{1- z^4}\\
                     & = \frac{1+ 3 i z (1- z^2) + 3 z^2  }{1 - z^4}
\end{align*}
At $z := ir  = i(0.253077)$, we have $zf_2'(z)/f_2(z)\approx 0$, thereby proving that the radius of starlikeness obtained for the class $\mathcal{K}_2$ is sharp.
\begin{enumerate}
  \item In order to compute $R_ {\mathcal{S}^*(\a)}$, we estimate the value of $r \in [0,1]$ satisfying
   \begin{equation*}
  \RE \frac{z f'(z)}{f(z)} \geq  \frac{1 - 3 r ( r^2 +  r + 1)}{(1-r^4)} \geq \a.
\end{equation*}
Therefore, the number $R_ {\mathcal{S}^*(\a)}$, is the root of the equation $\a r^4 - 3 r ( r^2 +  r + 1) +(1-\a)=0$ in $[0,1]$. For the function $f_2 \in \mathcal{K}_2$ given by \eqref{f2}, we have
\begin{align} \label{f22}
  \frac{z f'_2(z)}{f_2(z)} & = \frac{1 + 3 i z + 3 z^2 - 3 i z^3}{1- z^4}
\end{align}
At $z:= i r = {R}_{\mathcal{S}^*(\alpha)}$,  \eqref{f22} reduces to
\begin{align*}
  \frac{z f'_2(z)}{f_2(z)} & = \frac{1 - 3 r - 3 r^2 - 3 r^3}{1- r^4}
= \alpha,
\end{align*}
thereby proving that the radius is sharp.

\item We would use Lemma \ref{T1} to compute the lemniscate starlike radius for $f \in \mathcal{K}_2$. So, let $a = (1 + r^4)/(1-r^4)$. Then $1 \leq a < \infty$ for $r\in [0,1)$, and $a < \sqrt{2}$ when $r < \left(\left(\sqrt{2} - 1\right)/\left(\sqrt{2} + 1\right)\right)^{1/4}$. From \ref{2.6}, we know that $f\in \mathcal{K}_2$ maps the disk $|z| \leq r$ onto the disk
    \begin{equation}
  \left| \frac{z f'(z)}{f(z)} - \frac{1+r^4}{1-r^4} \right| \leq  \frac{r (r^3 + 3 r^2 + 3 r + 3)}{1-r^4}.
    \end{equation}
    So, consider
    \[
    \frac{r (r^3 + 3 r^2 + 3 r + 3)}{1-r^4} \leq \sqrt{2} - \frac{1+r^4}{1-r^4},
    \]
    and let $r^*$ be the smallest positive real root of the equation
    \[
    \left( \sqrt{2} + 2\right) r^4 - 4r^3 - 2r^2 - 4r + \left( 3 - 2\sqrt{2}\right) = 0.
    \]
    Then by Lemma \ref{T1}, the lemniscate starlike radius $R_{\mathcal{S}^*_L}$ for $f \in \mathcal{K}_2$ is given by
    \[
    R_{\mathcal{S}^*_L} = \min\left\{ \left(\frac{\sqrt{2} - 1}{\sqrt{2} + 1}\right)^{1/4}, r^*\right\} = r^* = \frac{\sqrt{2} -1}{\sqrt{2} + 2} = 0.12132\cdots
    \]
    This radius may not be sharp.

\item We use Lemma \ref{T2} to compute the parabolic starlike radius for $f \in \mathcal{K}_2$. For $a = (1 + r^4)/(1-r^4)$, we have $a \leq 3/2$ if $r \leq (1/5)^{1/4} \approx 0.668740305$. by Lemma \ref{T2}, consider
    \[
     \frac{r (r^3 + 3 r^2 + 3 r + 3)}{1-r^4} \leq \frac{1+r^4}{1-r^4} - \frac{1}{2},
    \]
    and let $r^*$ be the smallest positive real root of the equation
    \[
    r^4 - 6r^3 - 6r^2 - 6r + 1 = 0.
    \]
    Then the $\mathcal{S}_{P}$-radius is
    \[
    R_{\mathcal{S}_P} = \min \left\{ \left( \frac{1}{5}\right)^{1/4}, r^*\right\} = r^* \approx 0.1432698.
    \]

We see that sharpness follows for the function $f_2 \in \mathcal{K}_2$ defined in \eqref{f2}. As shown previously, at $z := i r $, we have
\begin{equation*}
   \frac{z f_2'(z)}{f_2(z)}  = \frac{1 - 3 r - 3 r^2 - 3 r^3}{1- r^4}.
\end{equation*}
Thus,
\begin{align*}
\left| \frac{zf_2'(z)}{f_2(z)}- 1  \right| &= \left| \frac{r (3 + 3 r + 3 r^2 - r^3)}{1 - r^4}\right|%\\
                                     %      &=  \frac{r (3 + 3 r + 3 r^2 - r^3)}{1 - r^4}
\end{align*}
For $z := i r =  i {R}_{\mathcal{S}_P}$, we have
\begin{align*}
  \RE  \frac{z f_2'(z)}{f_2(z)}  &= \frac{1 - 3 r - 3 r^2 - 3 r^3}{1- r^4} \quad(\approx 0.5)\\
&=\frac{r (3 + 3 r + 3 r^2 - r^3)}{1 - r^4}=  \left| \frac{zf_2'(z)}{f_2(z)}- 1  \right|.
\end{align*}
Thus the radius obtained is sharp for the function $f_2$.

\item For the $\mathcal{S}_e$-radius of $f\in \mathcal{K}_2$, we will use Lemma \ref{T3} since if $a = (1 + r^4)/(1-r^4)$, $0\leq r <1$, we have $a < e$ for $r < [(e-1)/(e+1)]^{1/4} \approx 0.82449$. Also, since $a \leq \frac{1}{2}(e + e^{-1})$ for $r < [(e-1)/(e+1)]^{2} \approx 0.213552$, consider
    \[
    \frac{r (r^3 + 3 r^2 + 3 r + 3)}{1-r^4} \leq \frac{1+r^4}{1-r^4} - \frac{1}{e},
    \]
    and let $r^*$ be the smallest positive real root of the equation
    \[
    r^4 - e(3r^3 + 3r^2 + 3r) + e - 1 = 0.
    \]
    Then the $\mathcal{S}_{e}$-radius is
    \[
    R_{\mathcal{S}_e} = \min \left\{ \left( \frac{e-1}{e+1}\right)^{2}, r^*\right\} = r^* \approx 0.174887.
    \]

\begin{equation*}
  \left|\log \frac{zf'_2(z)}{f_2(z)} \right| = \left| \log \frac{1 - 3 r - 3 r^2 - 3 r^3}{1- r^4}\right|=1,
\end{equation*}
 hence proving that the exponential starlike radius obtained for the class $\mathcal{K}_2$ is sharp.

\item By using Lemma \ref{T4}, it can be proven similarly as above that the $\mathcal{S}^*_c$-radius $R_{\mathcal{S}^*_c}$ for the class $\mathcal{K}_2$ is the smallest positive real root of the equation $9 r^3 + 9 r^2 + 9 r - 2 - r^4 = 0$, which is $R_{\mathcal{S}^*_c} \approx 0.182815$.

    The radius obtained is sharp for the function $f_2$ defined in \eqref{f2} as for $z := ir = i {R}_{\mathcal{S}^*_c}$,
\begin{equation*}
    \frac{zf'_2(z)}{f_2(z)} =  \frac{1 -3 r - 3 r^2 -3 r^3}{1 - r^4} =\frac{1}{3}=h_c(-1)\in \partial h_c(\mathbb{D}).
\end{equation*}%
%as illustrated in Figure \ref{fig4} below.
%\begin{figure}[h]
%  \includegraphics[width=4cm, height=5cm]{11}
% \caption{Sharpness of cardioid starlikeness} \label{fig4}
%\end{figure}

\item  The $\mathcal{S}^*_{\leftmoon}$-radius ${R}_{\mathcal{S}^*_{\leftmoon}}$ for the class $\mathcal{K}_2$ is the smallest positive real root of the equation $r^4 (1- \sqrt{2}) + 3 r^3 + 3 r^2 + 3 r = 2 - \sqrt{2}$. This can be obtained by considering the inequality
    \[
    \frac{3 r^3 +3 r^2 + 3 r - 1}{1- r^4} \leq 1- \sqrt{2}
    \]
    and then using Lemma \ref{T7}.

    The radius is sharp for the function $f_2$ defined in \eqref{f2}, since at $z := i r = i{R}_{\mathcal{S}^*_{\leftmoon}}$, we have
\begin{align*}
  \left|  \left( \frac{zf_2'(z)}{f_2(z)}\right)^2 - 1 \right| &= \left| \left( \frac{1 - 3 r - 3 r^2 - 3 r^3}{1- r^4}\right)^2 -1 \right| \quad (\approx 0.82842)\\
  &= 2\left| \frac{1 - 3 r - 3 r^2 - 3 r^3}{1- r^4} \right| = 2 \left|\frac{zf_2'(z)}{f_2(z)} \right|.
\end{align*}

\item In order to find the $\mathcal{S}^*_{\sin}$-radius for the function $f \in \mathcal{K}_2$, we make use of Lemma \ref{T5}. It is easy to see that $ 1 - \sin 1 < a = {(1+r^4)}/{(1-r^4)} < 1+ \sin 1 $ for $r < [(\sin{1})/(2 + \sin{1})]^{1/4}$. Since $a>1$, consider
    \begin{equation*}\label{11}
     \frac{r (r^3 + 3 r^2 + 3 r + 3)}{1-r^4} \leq \sin 1 - \left(\frac{1+r^4}{1-r^4} - 1\right).
    \end{equation*}
    Then the $\mathcal{S}^*_{\sin}$-radius, $R_{\mathcal{S}^*_{\sin}} (\approx 0.219049)$, is the smallest positive real root of the equation
    \[
    (3 + \sin{1})r^4 + 3r(r^2 + r +1) = \sin{1}.
    \]
The radius obtained is sharp for the function $f_2$ defined in \eqref{f2}.% and %it can be seen in Figure \ref{fig5}.
%\begin{figure}[h]
%  \includegraphics[width=6cm, height=6cm]{12}
% \caption{Sharpness of sine starlikeness} \label{fig5}
%\end{figure}

\item We use Lemma \ref{T8} in order to compute the $\mathcal{S}^*_R$- radius for the class $\mathcal{K}_2$. Since $2 (\sqrt{2}-1) < a = {(1+ r^4)}/{(1- r^4)}\leq \sqrt{2}$ for $r < [(\sqrt{2} - 1)/(\sqrt{2} + 1)]^{1/4} \approx 0.64359$, by Lemma \ref{T8}, we consider
    \[
    \frac{r(r^3 + 3 r^2 +3 r + 3)}{1-r^4} \leq  \frac{1 + r^4}{1 - r^4} - 2(\sqrt{2} - 1).
    \]
    Then the $\mathcal{S}^*_R$- radius for $\mathcal{K}_2$ can be computed to be $R_{\mathcal{S}^*_{R}} \approx 0.0870259$. The radius obtained is sharp for the function $f_2$ defined in \eqref{f2}. %and it is illustrated in Figure \ref{fig6}.
%\begin{figure}[h]
%  \includegraphics[width=7cm, height=6cm]{13}
% \caption{Sharpness of a rational function starlikeness} \label{fig6}
%\end{figure}
%
Indeed, for the function $f_2$ defined in \eqref{f2}, we have
at $z := ir = i {R}_{\mathcal{S}^*_R}$,
\begin{equation*}
    \frac{zf'_2(z)}{f_2(z)} =  \frac{1 - 3 r - 3 r^2 -3 r^3}{1 - r^4} =2 \sqrt{2} - 2=h_R(-1)\in \partial h_R(\mathbb{D}).
\end{equation*}
This shows that the result is sharp.

\item Finally, for the $\mathcal{S}^*_{RL}$- radius, $R_{\mathcal{S}^*_{RL}}$, for the class $\mathcal{K}_2$, by Lemma \ref{T9}, the value of $R_{\mathcal{S}^*_{RL}} \approx 0.0541073$ is obtained from solving the equation
    \begin{align*}
    (1- r^4) \big\{  (1- r^4)^2  - \big( (\sqrt{2} - 1)& - (\sqrt{2} + 1) r^4 \big)^2 \big\}^{1/2} =  (r^4 + 3r^3 + 3r^2 + 3r)^2 \\& \quad +(1- r^4)^2 - \big((\sqrt{2} - 1)-  (\sqrt{2} + 1) r^4 \big)^2. \qedhere
    \end{align*}

\end{enumerate}
\end{proof}
\bigskip

The last theorem aims at computing the various radii of starlikeness for the function $f \in \mathcal{K}_3$ that satisfies  $\RE  ((1-z^2)f(z)/z) >0$.
Consider the function $f_3: \mathbb{D}\rightarrow \mathbb{C}$ defined by
\begin{equation} \label{f3}
  f_3(z) = \frac{z(1+ i z)}{(1-z^2)(1- i z)}
\end{equation}
Clearly,
\begin{equation*}
  \RE  \frac{(1-z^2)}{z}f_3(z)  = \RE  \frac{1+ i z}{1- i z} >0.
\end{equation*}
Therefore, the function $f_3 \in \mathcal{K}_3$ and $\mathcal{K}_3 \neq \phi$. This function $f_3$ would serve as an extremal function for various radius problems in the following theorem.

\begin{theorem}
 For $f \in \mathcal{K}_3$, the following results hold:
\begin{enumerate}
  \item The sharp $\mathcal{S}^*(\a)$ radius is the smallest positive real root of the equation $ (1+\a) r^4 - 2 r (r^2 + r + 1) + (1- \a)=0$, \quad $0 \leq \a <1$.
  \item The sharp $\mathcal{S}^*_L$ radius is ${R}_{\mathcal{S_L}}= (\sqrt{2}-1)/(\sqrt{2}+1) \approx 0.171573$.
  \item The sharp $\mathcal{S}_{P}$ radius is the smallest positive real root of the equation $ 4 r^3 + 4 r^2 + 4 r -1  - 3r^4   = 0$ i.e. ${R}_{\mathcal{S}_{P}} \approx 0.2021347$.
  \item The sharp $\mathcal{S}^*_{e}$ radius is the smallest positive real root of the equation $(2 r^3 + 2 r^2 + 2 r -1 - r^4)e + 1 - r^4 = 0$ i.e.  ${R}_{\mathcal{S}^*_{e}} \approx 0.244259$.
  \item The sharp $\mathcal{S}^*_{c}$ radius is the smallest positive real root of the equation $  3 r^3 + 3 r^2 + 3 r - 1 - 2 r^4 = 0 $ i.e. ${R}_{\mathcal{S}^*_{c}} \approx 0.254726$.
\item The sharp $\mathcal{S}^*_{\leftmoon}$ radius is the smallest positive real root of the equation $  2 r^3 + 2 r^2 + 2 r - \sqrt{2} r^4  = 2 - \sqrt{2}$ i.e. ${R}_{\mathcal{S}^*_{\leftmoon}} \approx 0.229877$.
  \item The sharp $\mathcal{S}^*_{\sin}$ radius is ${R}_{\mathcal{S}^*_{\sin}}= \sin{1}/(2 + \sin{1}) \approx 0.296139$.
\item The sharp $\mathcal{S}^*_{R}$ radius is the smallest positive real root of the equation $ r^4 + 2 r^3 + 2 r^2 +  2 r  -3 + 2 \sqrt{2} (1 - r^4)=0  $ i.e. ${R}_{\mathcal{S}^*_{R}} \approx 0.0790749$.
  \item The $\mathcal{S}^*_{RL}$ radius is ${R}_{\mathcal{S}^*_{RL}} \approx 0.125145$.
\end{enumerate}
\end{theorem}

\begin{proof}
Let the function $f \in \mathcal{K}_3$. Then
\begin{equation}\label{3.1}
  \RE{\Big( \frac{1-z^2}{z} f(z) \Big)}>0 \quad (z \in \mathbb{D}).
\end{equation}
Define $p: \mathbb{D}\rightarrow \mathbb{C}$ as
\begin{equation}\label{3.2}
  p(z)= \frac{1-z^2}{z} f(z)
\end{equation}
Therefore, by equation \eqref{3.1}, we have $p \in \mathcal{P}$ and %equation \eqref{3.2} yields
\begin{equation*}
  f(z) = \frac{z}{(1-z^2)} p(z).
\end{equation*}
From this, take logarithm on both sides and then differentiate with respect to $z$:
\begin{equation}\label{3.3}
  \frac{z f'(z)}{f(z)}= \frac{1+z^2}{1-z^2}+ \frac{z p'(z)}{p(z)}.
\end{equation}
By using equations \eqref{4},\eqref{5} and \eqref{3.3}, we can prove that the function $f$ maps the disk $|z|\leq r$ onto the disk
\begin{equation}\label{3.6}
  \left| \frac{z f'(z)}{f(z)} - \frac{1+r^4}{1-r^4} \right| \leq  \frac{2 r (r^2 +  r + 1)}{1-r^4}.
\end{equation}
In order to solve radius problems, we are interested in computing the value of $r$ for which the disk in \eqref{3.6} is contained in the corresponding regions.
Again the classes we are considering here are all subclasses of starlike functions and therefore, are defined by the quantity $z f'(z)/f(z)$ lying in some region in the right half plane. In particular, for $f$ to be in $\mathcal{S^*}$, we need
\begin{equation*}
  \RE \frac{z f'(z)}{f(z)} \geq  \frac{1 + r^4 - 2 r ( r^2 +  r + 1)}{1-r^4} \geq 0.
\end{equation*}
Thus the function $f \in \mathcal{K}_3$ is starlike in $|z|\leq 0.346014$. With this, now all the radii we estimate here shall be less than $0.346014$.

For the function $f_3$ defined in \eqref{f3}, we have
\begin{align}\label{f33}
 \frac{zf'_3(z)}{f_3(z)} & = \frac{1 + 2 i z + 2 z^2 - 2 i z^3 + z^4}{1- z^4}\\ \notag
                     & = \frac{1+ 2 i z (1- z^2) + 2 z^2 + z^4 }{1 - z^4}.
\end{align}
At $z := ir = i(0.346014)$, we have $zf'_3(z)/f_3(z)\approx 0$, thereby proving that the radius of starlikeness obtained for the class $\mathcal{K}_3$ is sharp.
\begin{enumerate}
  \item To determine the radius $R_ {\mathcal{S}^*(\a)}$ of starlikeness of order $\alpha$,  we estimate the value of $r \in [0,1]$ satisfying
   \begin{equation*}
  \RE \frac{z f'(z)}{f(z)} \geq  \frac{1 + r^4 - 2 r ( r^2 +  r + 1)}{1-r^4} \geq \a.
    \end{equation*}
 Hence, $R_ {\mathcal{S}^*(\a)}$ is the root of the equation $(1 + \a) r^4 - 2 r ( r^2 +  r + 1) +(1-\a)=0 $ in $[0,1]$. From \eqref{f33}, if $z:= i r = i{R}_{\mathcal{S}^*(\alpha)}$, then \eqref{f33} reduces to
\begin{align*}
  \frac{z f_3'(z)}{f_3(z)} & = \frac{1 - 2 r - 2 r^2 - 2 r^3 + r^4}{1- r^4}
= \alpha,
\end{align*}
which shows that $f_3$ is the extremal function.

\item We can use Lemma \ref{T1} to compute the lemniscate starlike radius for $f \in \mathcal{K}_3$. For $a = (1 + r^4)/(1-r^4)$, we have $a < \sqrt{2}$ when $r < \left(\left(\sqrt{2} - 1\right)/\left(\sqrt{2} + 1\right)\right)^{1/4}$. By \eqref{2.6} and Lemma \ref{T1}, consider
    \[
    \frac{2r (r^2 + r + 1)}{1-r^4} \leq \sqrt{2} - \frac{1+r^4}{1-r^4},
    \]
    and let $r^*$ be the smallest positive real root of the equation
    \[
    \left( \sqrt{2} + 1\right) r^4 + 2r^3 + 2r^2 + 2r + \left( 1 - \sqrt{2}\right) = 0.
    \]
     Hence, the lemniscate starlike radius $R_{\mathcal{S}^*_L}$ for $f \in \mathcal{K}_3$ is given by
    \[
    R_{\mathcal{S}^*_L} = \min\left\{ \left(\frac{\sqrt{2} - 1}{\sqrt{2} + 1}\right)^{1/4}, r^*\right\} = r^* = \frac{1 - \sqrt{2}}{1 + \sqrt{2}} = 0.1715728753\ldots.
    \]

    For the sharpness, consider the function $\hat{f_3}: \mathbb{D}\rightarrow \mathbb{C}$ defined by
    \begin{equation} \label{f_31}
    \hat{ f_3}(z) = \frac {z}{(1+z)^2}.
    \end{equation}
    Clearly,
    \begin{equation*}
        \RE  \frac{1-z^2}{z}\hat{f_3}(z)  = \RE  \frac{1 - z}{1+  z} >0.
    \end{equation*}
    So $\hat{f_3} \in \mathcal{K}_3$. Also
    \begin{equation*}
    \left| \left(\frac{z\hat{f_3}'(z)}{\hat{f_3}(z)}\right)^2 -1  \right| = \left| \left(\frac{1 - z}{1 + z}\right)^2 -1 \right|.
    \end{equation*}
    Now, for $z := r = {R_{\mathcal{S}^*_L}}$, we have $(1 - z)/(1 + z) = \sqrt{2}$ and
    \begin{equation*}
    \left| \left(\frac{z\hat{f_3}'(z)}{\hat{f_3}(z)}\right)^2 -1  \right| = \left| \left(\sqrt{2}\right)^2 -1 \right|=1,
    \end{equation*}
    Therefore, the radius obtained is sharp for the function $\hat{f_3}$.

\item For the parabolic starlike radius for $f \in \mathcal{K}_3$, we would use Lemma \ref{T2}. For $r \leq (1/5)^{1/4}$, we have $a= {(1+r^4)}/{(1-r^4)}\leq 3/2$, and if we consider
    \begin{equation}\label{39}
   \frac{ 2 r ( r^2 +  r + 1)}{1-r^4} \leq \frac{1+r^4}{1-r^4} -\frac{1}{2},
    \end{equation}
    then the $\mathcal{S}_P$-radius is given by
    \[
    R_{\mathcal{S}_P} = \min\left\{ \left(\frac{1}{5}\right)^{1/4}, r^*\right\} = r^* \approx 0.2021347,
    \]
    where $r^*$ is the smallest positive real root of the equation $3r^4 - 4 r^3 - 4 r^2 - 4 r + 1 = 0$.

    The sharpness of the result follows for the function $f_3$ defined in \eqref{f3}. As shown previously, at $z := i r$, we have
\begin{equation*}
   \frac{z f_3'(z)}{f_3(z)}  = \frac{1 - 2 r - 2 r^2 - 2 r^3 + r^4}{1- r^4}.
\end{equation*}
%Thus,
%\begin{align*}
%\left| \frac{zf_3'(z)}{f_3(z)}- 1  \right| &= \left| \frac{2 r (1 + r + r^2 - r^3)}{r^4 -1)}\right|\\
%                                           &=  \frac{2 r (1 + r + r^2 - r^3)}{1 - r^4}
%\end{align*}
Then for $z := i r = {R}_{\mathcal{S}_P}$, we have
\begin{align*}
  \RE  \frac{z f'_3(z)}{f_3(z)}  &= \frac{1 - 2 r - 2 r^2 - 2 r^3 + r^4}{1- r^4} \quad (= 0.5)\\
&= \frac{2 r (1 + r + r^2 - r^3)}{1 - r^4}=  \left| \frac{zf_3'(z)}{f_3(z)}- 1  \right|,
\end{align*}
thus illustrates the radius obtained is sharp for the function $f_3$.

\item By using Lemma \ref{T3} and considering
\begin{equation*}
 \frac{2 r ( r^2 +  r + 1)}{(1-r^4)} \leq \frac{1+r^4}{1-r^4} -\frac{1}{e},
\end{equation*}
it can be proven similarly as above that the exponential starlike radius $R_{\mathcal{S}^*_e}$ for the class $\mathcal{K}_3$ is the smallest positive real root of the equation $(2 r^3 + 2 r^2 + 2 r -1 - r^4)e + 1 - r^4 = 0$. Again, for the function $f_3$ given in \eqref{f3}, at  $z := i r = i{R}_{\mathcal{S}^*_e} \approx i(0.244259)$, we have
\begin{equation*}
  \left|\log \frac{zf'_3(z)}{f_3(z)} \right| = \left| \log \frac{1 - 2 r - 2 r^2 - 2 r^3 + r^4}{1- r^4}\right|=1,
\end{equation*}
thereby proving that the result obtained is sharp.

\item For the $\mathcal{S}^*_c$-radius for the class $\mathcal{K}_3$, we use Lemma \ref{T4} by considering
        \begin{equation*}
        \frac{2 r ( r^2 +  r + 1)}{1-r^4} \leq \frac{1+r^4}{1-r^4} -\frac{1}{3}.
        \end{equation*}
       Then it can be proven similarly as above that $R_{\mathcal{S}^*_c}$ is the smallest positive real root of the equation $$3 r^3 + 3 r^2 + 3 r - 1 - 2 r^4 = 0.$$
       The radius obtained is sharp for the function $f_3$ defined in \eqref{f3}.% and it can be seen in Figure \ref{fig7}.
%\begin{figure}[h]
%  \includegraphics[width=4cm, height=4cm]{21}
% \caption{Sharpness of cardioid starlikeness} \label{fig7}
%\end{figure}

Indeed, for the function $f_3$, we have
at $z := i r = i {R}_{\mathcal{S}^*_c} \approx i(0.254716)$ that
\begin{equation*}
    \frac{zf'_3(z)}{f_3(z)} =  \frac{1 - 2 r - 2r^2 -2 r^3 +r^4}{1 - r^4} =\frac{1}{3}=h_c(-1)\in \partial h_c(\mathbb{D}).
\end{equation*}

\item  By proving similarly as above, the $\mathcal{S}^*_{\leftmoon}$-radius for the class $\mathcal{K}_3$ is the smallest positive real root of the equation
    \[ 2 r^3 + 2 r^2 + 2 r - \sqrt{2} r^4 = 2 - \sqrt{2}.\]
    In this case, we would use Lemma \ref{T7} and consider
    \begin{equation*}
     \frac{2 r ( r^2 +  r + 1)}{1- r^4} \leq 1- \sqrt{2} + \frac{1 + r^4}{1 - r^4}.
    \end{equation*}
    The radius is sharp for the function $f_3$ defined in \eqref{f3}, since at $z := ir = i{R}_{\mathcal{S}^*_{\leftmoon}}$, we have
\begin{align*}
  \left|  \left( \frac{zf_3'(z)}{f_3(z)}\right)^2-1 \right| &= \left| \left( \frac{1 - 2 r - 2 r^2 - 2 r^3 + r^4}{1- r^4}\right)^2 -1 \right| \\&= 2\left| \frac{1 - 2 r - 2 r^2 - 2 r^3 + r^4}{1- r^4} \right| = 2 \left|\frac{zf_3'(z)}{f_3(z)} \right|.
\end{align*}

\item In order to find the $\mathcal{S}^*_{\sin}$-radius for the function $f \in \mathcal{K}_3$, we make use of Lemma \ref{T5}, where we would consider
    \begin{equation*}\label{11}
    \frac{2 r ( r^2 +  r + 1)}{(1-r^4)} \leq \sin 1 - \frac{2 r^4}{1-r^4}.
    \end{equation*}
    The $\mathcal{S}^*_{\sin}$-radius, $R_{\mathcal{S}^*_{\sin}}$, is smallest positive real root of the equation
    \[
    2r(r^3 + r^2 + r + 1) = (\sin{1})(1 - r^4).
    \]
    The radius obtained is sharp for the function $f_3$ defined in \eqref{f3}.  
    %and it can be seen in Figure \ref{fig8}:
%\begin{figure}[h]
%  \includegraphics[width=4cm, height=5cm]{22}
% \caption{Sharpness of sine starlikeness} \label{fig8}
%\end{figure}

\item We use Lemma \ref{T8} to compute the $\mathcal{S}^*_R$- radius for the class $\mathcal{K}_3$. By considering
    \begin{equation*}
\frac{2 r ( r^2 +  r + 1)}{1-r^4} \leq \frac{1 + r^4}{1 - r^4} - 2(\sqrt{2} - 1),
\end{equation*}
we would obtain $R_{\mathcal{S}^*_{R}}$ to be given by the smallest positive real root of the equation
    \[
    (2\sqrt{2} - 1)r^4 - 2r(r^2 + r + 1) + (3 - 2\sqrt{2}) = 0.
    \]

The radius obtained is sharp for the function $f_3$ defined in \eqref{f3}.

% and is shown in Figure \ref{fig9}:
%\begin{figure}[h]
%  \includegraphics[width=6cm, height=5cm]{23}
% \caption{Sharpness of rational starlikeness} \label{fig9}
%\end{figure}

Indeed, for the function $f_3$, we have at $z := i r = i {R}_{\mathcal{S}^*_c} \approx 0.0790749$,
\begin{equation*}
    \frac{zf'_3(z)}{f_3(z)} =  \frac{1  - 2 r - 2r^2 -2 r^3 + r^4}{1 - r^4} = 2 \sqrt{2}-2=h_R(-1)\in \partial h_R(\mathbb{D}).
\end{equation*}

\item Lastly, the $\mathcal{S}^*_{RL}$- radius for the class $\mathcal{K}_3$ is obtained by using Lemma \ref{T9} and from the equation
    \begin{align*}
 (1- r^4) \big\{  (1- r^4)^2  - \big( (\sqrt{2} - 1) - (\sqrt{2} + 1) r^4 \big)^2 \big\}^{1/2} &=  (2r^3+2r^2+2r)^2 +(1- r^4)^2 \\& \quad - \big((\sqrt{2} - 1)-  (\sqrt{2} + 1) r^4 \big)^2. \qedhere
\end{align*}

\end{enumerate}
\end{proof}

\subsection*{\textbf{Acknowledgment.}} The second author gratefully acknowledges support from USM research university grants 1001.PMATHS.8011101. %and 1001.PMATHS.8011038.

%\bibliographystyle{ref1}
%\bibliography{parallel-stochastic-hmm}

\begin{thebibliography}{99}
\bibitem{ali} R. M. Ali\, and\ V. Ravichandran, Uniformly convex and starlike functions, Ramanujan Mathematics Newsletter, {\bf 21} (2011), no.~1, 16--30.

\bibitem{Ali} R. M. Ali, N. K. Jain\ and\ V. Ravichandran, On the radius constants for classes of analytic functions, Bull. Malays. Math. Sci. Soc. (2) {\bf 36} (2013), no.~1, 23--38.

\bibitem{jain} R. M. Ali, N. K. Jain\ and\ V. Ravichandran, Radii of starlikeness associated with the lemniscate of Bernoulli and the left-half plane, Appl. Math. Comput. {\bf 218} (2012), no.~11, 6557--6565.

\bibitem{cho} N. E. Cho\ et al., Radius problems for starlike functions associated with the sine function, Bull. Iranian Math. Soc. {\bf 45} (2019), no.~1, 213--232.

\bibitem{gandhi} S. Gandhi\ and\ V. Ravichandran, Starlike functions associated with a lune, Asian-Eur. J. Math. {\bf 10} (2017), no.~4, 1750064, 12 pp.

\bibitem{kumar} S. Kumar\ and\ V. Ravichandran, A subclass of starlike functions associated with a rational function, Southeast Asian Bull. Math. {\bf 40} (2016), no.~2, 199--212.

\bibitem{Mac} T. H. MacGregor, The radius of univalence of certain analytic functions, Proc. Amer. Math. Soc. {\bf 14} (1963), 514--520.

\bibitem{Mac1} T. H. MacGregor, The radius of univalence of certain analytic functions. II, Proc Amer. Math. Soc. {\bf 14} (1963), 521--524.

\bibitem{mam} W. C. Ma\ and\ D. Minda, A unified treatment of some special classes of univalent functions, in {\it Proceedings of the Conference on Complex Analysis (Tianjin, 1992)}, 157--169, Conf. Proc. Lecture Notes Anal., I, Int. Press, Cambridge,MA.

\bibitem{mam1} W. C. Ma\ and\ D. Minda, Uniformly convex functions, Ann. Polon. Math. {\bf 57} (1992), no.~2, 165--175.


\bibitem{sumit} R. Mendiratta, S. Nagpal\ and\ V. Ravichandran, On a subclass of strongly starlike functions associated with exponential function, Bull. Malays. Math. Sci. Soc. {\bf 38} (2015), no.~1, 365--386.

\bibitem{sumit2} R. Mendiratta, S. Nagpal\ and\ V. Ravichandran, A subclass of starlike functions associated with left-half of the lemniscate of Bernoulli, Internat. J. Math. {\bf 25} (2014), no.~9, 1450090, 17 pp.

\bibitem{uni} K. S. Padmanabhan\ and\ R. Parvatham, Some applications of differential subordination, Bull. Austral. Math. Soc. {\bf 32} (1985), no.~3, 321--330.

\bibitem{raina} R. K. Raina\ and\ J. Sok\'{o}\l, Some properties related to a certain class of starlike functions, C. R. Math. Acad. Sci. Paris {\bf 353} (2015), no.~11, 973--978.

\bibitem{rob} M. S. Robertson, Certain classes of starlike functions, Michigan Math. J. {\bf 32} (1985), no.~2, 135--140.

\bibitem{ronning} F. R\o nning, A survey on uniformly convex and uniformly starlike functions, Ann. Univ. Mariae Curie-Sk\l odowska Sect. A {\bf 47} (1993), 123--134.

\bibitem{asha} A. Sebastian\ and\ V. Ravichandran, Radius of starlikeness of certain analytic functions, Math. Slovaca (accepted)

\bibitem{shah} G. M. Shah, On the univalence of some analytic functions, Pacific J. Math. {\bf 43} (1972), 239--250.

\bibitem{shan}T. N. Shanmugam\ and\ V. Ravichandran, Certain properties of uniformly convex functions, in {\it Computational methods and function theory 1994 (Penang)}, 319--324, Ser. Approx. Decompos., 5, World Sci. Publ., River Edge, NJ.

\bibitem{sok} J. Sok\'{o}l\ and\ J. Stankiewicz, Radius of convexity of some subclasses of strongly starlike functions, Zeszyty Nauk. Politech. Rzeszowskiej Mat. No. 19 (1996), 101--105.

\bibitem{sok2} J. Sok\'{o}l, Radius problems in the class $\mathcal{S}^*_{L}$, Appl. Math. Comput. {\bf 214} (2009), no.~2, 569–573.

\bibitem{ganga} A. Gangadharan, V. Ravichandran\ and\ T. N. Shanmugam, Radii of convexity and strong starlikeness for some classes of analytic functions, J. Math. Anal. Appl. {\bf 211} (1997), no.~1, 301--313.

\bibitem{sharma}K. Sharma, N. K. Jain\ and\ V. Ravichandran, Starlike functions associated with a cardioid, Afr. Mat. {\bf 27} (2016), no.~5-6, 923--939.

%\bibitem{sok2} J. Sok\'{o}l, Radius problems in the class $\mathcal{S}^*_{L}$, Appl. Math. Comput. {\bf 214} (2009), no.~2, 569–573.
%
%\bibitem{sok} J. Sok\'{o}l\ and\ J. Stankiewicz, Radius of convexity of some subclasses of strongly starlike functions, Zeszyty Nauk. Politech. Rzeszowskiej Mat. No. 19 (1996), 101--105.
%
%\bibitem{ura} B. A. Uralegaddi, M. D. Ganigi\ and\ S. M. Sarangi, Univalent functions with positive coefficients, Tamkang
%J. Math.{\bf 25} (1994), no.~3, 225–230.
%\bibliography{ref1}
%\bibliographystyle{siam}
\end{thebibliography}

\end{document}